\documentclass{amsart}
\usepackage[active]{srcltx}
\usepackage{amsmath}\usepackage{amssymb}\usepackage{amscd}\usepackage{amsthm}\usepackage{amsfonts}
\usepackage{graphicx}
\usepackage{cite}

\newtheorem{theorem}{\textbf{Theorem}}[section]
\newtheorem{corollary}[theorem]{\textbf{Corollary}}

\newtheorem*{theo-intro}{\textbf{Theorem}}

\theoremstyle{definition}
\newtheorem{definition}[theorem]{\textbf{Definition}}

\theoremstyle{remark}
\newtheorem{remark}[theorem]{\textbf{Remark}}
\newtheorem{remarks}[theorem]{\textbf{Remarks}}

\newcounter{pic}

\numberwithin{equation}{section}

\DeclareRobustCommand\mmodels{\Relbar\joinrel\mathrel{|}}
\def\C{\mathbb{C}}
\def\Z{\mathbb{Z}}

\def\og{\overline{g}}

\newcommand{\Comp}{\operatorname{Comp}}

\newcommand{\Mat}{\operatorname{Mat}}

\def\cP{\mathcal{P}}

\begin{document}

\title{The HOMFLYPT polynomials of sublinks and the Yokonuma--Hecke algebras}

\author{L. Poulain d'Andecy}
\email{loic.poulain-dandecy@univ-reims.fr}
\address{Universit\'e de Reims Champagne-Ardenne, UFR Sciences exactes et naturelles, Laboratoire de Math\'ematiques EA 4535 Moulin de la Housse BP 1039, 51100 Reims, France}
\author{E. Wagner}
\email{emmanuel.wagner@u-bourgogne.fr}
\address{Universit\'e de Bourgogne-Franche-Comt\'e, Institut de mathŽmatiques de Bourgogne UMR 5584, 21000 Dijon, France}





\begin{abstract}
We describe completely the link invariants constructed using Markov traces on the Yokonuma--Hecke algebras in terms of the linking matrix and the HOMFLYPT polynomials of sublinks.
\end{abstract}

\maketitle

\section{Introduction}

 The Yokonuma--Hecke algebras (of type GL), denoted $Y_{d,n}$, have been used by J. Juyumaya and S. Lambropoulou to construct invariants for various types of links, in the same spirit as the construction of the HOMFLYPT polynomial from usual Hecke algebras. We refer to \cite{ChJKL} and references therein. In particular, the algebras $Y_{d,n}$ provide invariants for classical links.

The invariants coming from Yokonuma--Hecke algebras were obtained by a different method in \cite{JaPo}. In this approach, the starting point was the fact that the algebra $Y_{d,n}$ is isomorphic to a direct sum of matrix algebras with coefficients in tensor products of usual Hecke algebras. This allowed an explicit description of the space of Markov traces on $\{Y_{d,n}\}_{n\geq1}$, in terms of the known Markov trace on Hecke algebras. This in turn provided a whole space of invariants. The Juyumaya--Lambropoulou invariants are precisely identified among this family \cite[\S 6.5]{JaPo} (see also \cite{Po}).

Within the approach via the isomorphism, the study of the invariants coming from the Yokonuma--Hecke algebras is reduced to the understanding of the isomorphism. More precisely, once the isomorphism is properly understood (especially on images of braids), the invariants from the Yokonuma--Hecke algebras can be expressed explicitly as combinations of HOMFLYPT polynomials of some links. This is how the main result below is obtained.

The Markov traces on $\{Y_{d,n}\}_{n\geq1}$ were classified in \cite{JaPo}. For each $d>0$, the space of Markov traces is of dimension $2^d-1$. It was further proved in \cite{Po} that every invariant for links obtained via a Markov trace on $\{Y_{d,n}\}_{n\geq1}$ can be expressed in terms of  a family of traces $\{T_d(L)\}_{d>0}$. In this sense, it is enough to consider the set of invariants $\{T_d(L)\}_{d>0}$. Note that $T_d(L)$ was denoted $P^{d,\{1,\dots,d\}}_L$ in \cite{Po}.

The main result of this paper is that the invariants $\{T_d(L)\}_{d>0}$ coming from Yokonuma--Hecke algebras  $Y_{d,n}$ are described entirely in terms of linking numbers and HOMFLYPT polynomials of sublinks.
\begin{theo-intro}
Let $L$ be a link and $d>0$. We have
\begin{equation}
T_d(L)=d!\sum_{\{L_1,\dots,L_d\}}(\gamma')^{l(L_1,\dots,L_d)}P(L_1)\dots P(L_d) \in \C[u^{\pm1},v^{\pm1},\gamma'^{\pm1}],
\end{equation}
where the sum is over the set of all complete families $\{L_1,\dots,L_d\}$ of $d$ distinct non-empty sublinks of $L$.
\end{theo-intro}
In the formula above, $l(L_1,\dots,L_d)$ is the sum of  the linking numbers of the family and $P(L)$ is the HOMFLYPT polynomial of a link $L$. The parameters $u,v$ are the two parameters of the HOMFLYPT polynomial, while $\gamma'$ is an additional parameter specific to the construction of invariants from the Yokonuma--Hecke algebras. The Juyumaya--Lambropoulou invariants are certain linear combinations of the invariants $T_d(L)$, for some specialisations of the parameter $\gamma'$ \cite{Po}. For these invariants, a similar result has been proved independently by Lickorish using diagrammatic techniques (see \cite[Appendix]{ChJKL}).

Notice first that the invariant $T_d(L)$ is only non trivial if $d$ is smaller than the number of connected components of $L$. Secondly the previous invariant was obtained via the study of Markov traces on the Yokonuma--Hecke algebras but a posteriori it is obviously an invariant. What is less clear while looking at the formula above is that it can be provided by Markov traces factoring  through $Y_{d,n}$. Finally, one can imagine defining very similar algebras starting for very similar formulas defined using other invariants than the HOMFLYPT invariants, for instance, the $2$-variable  Kauffman polynomial.

Besides, we indicate that the theorem admits a straightforward analogue for affine Yokonuma--Hecke algebras, which can be proved following exactly the same arguments than in this paper. In the affine case, the approach to invariants of links in the solid torus via an isomorphism theorem was explained in \cite{Po} (generalisations of Juyumaya--Lambropoulou invariants were obtained in \cite{ChPo}). Instead of the HOMFLYPT polynomials, one would obtain their analogues for links in the solid torus from \cite{La}. For clarity of the exposition, we present here only the non-affine case.

\section{Yokonuma--Hecke algebras and isomorphism theorem}\label{sec-def}

Let  $d,n \in\Z_{>0}$ and $u$ and $v$ be indeterminates. We work over the ring $\C[u^{\pm1},v]$.

\subsection{Definitions}

We use $\mathfrak{S}_n$ to denote the symmetric group on $n$ elements, and $s_i$ to denote the transposition $(i,i+1)$. The Yokonuma--Hecke algebra $Y_{d,n}$ is generated by elements
$$g_1,\ldots,g_{n-1},t_1,\ldots,t_n,$$
subject to the following defining relations (\ref{def-1})--(\ref{def-3}):
\begin{equation}\label{def-1}
\begin{array}{rclcl}
g_ig_j & = & g_jg_i &\  & \mbox{for $i,j=1,\ldots,n-1$ such that $\vert i-j\vert > 1$,}\\[0.1em]
g_ig_{i+1}g_i & = & g_{i+1}g_ig_{i+1} && \mbox{for $i=1,\ldots,n-2$,}
\end{array}
\end{equation}
\begin{equation}\label{def-2}
\hspace{-1.2cm}\begin{array}{rclcl}
t_it_j & =  & t_jt_i &\qquad&  \mbox{for $i,j=1,\ldots,n$,}\\[0.1em]
g_it_j & = & t_{s_i(j)}g_i && \mbox{for $i=1,\ldots,n-1$ and $j=1,\ldots,n$,}\\[0.1em]
t_j^d   & =  &  1 && \mbox{for $j=1,\ldots,n$,}
\end{array}
\end{equation}
\begin{equation}\label{def-3}
\hspace{-3.2cm}
\begin{array}{rclcl}
g_i^2  & = & u^2 + u\,v \, e_{i} \, g_i &\quad& \mbox{for $i=1,\ldots,n-1$,}
\end{array}
\end{equation}
where $e_i :=\displaystyle\frac{1}{d}\sum_{1\leq s\leq d}t_i^s t_{i+1}^{-s}$. The elements $e_i$ are idempotents and we have:
\begin{equation}
g_i^{-1} = u^{-2}g_i - u^{-1}v\, e_i  \qquad \mbox{for all $i=1,\ldots,n-1$}.
\end{equation}

\paragraph{\textbf{Hecke algebra $H_n$.}}  The Hecke algebra $H_n$ is isomorphic to the quotient of $Y_{d,n}$ (for any $d>0$) by the relations $t_j=1$, $j=1,\dots,n$. We denote by $\overline{\,\cdot\,}$ the corresponding surjective morphism from $Y_{d,n}$ to $H_n$. 

The algebra $H_n$ has generators denoted $\og_1,\dots,\og_{n-1}$ and as defining relations, the images of relations (\ref{def-1}) and (\ref{def-3}) by $\overline{\,\cdot\,}$. In particular, we have
\[\og_i^2  =  u^2 + u\,v \, \og_i \qquad\text{for $i=1,\ldots,n-1$ .}\]

\paragraph{\textbf{Compositions of $n$.}}

Let $\operatorname{Comp}_d (n)$ be the set of {\it $d$-compositions} of $n$, that is the set of $d$-tuples $\mu=(\mu_1,\ldots,\mu_d)\in\Z_{\geq0}^d$ such that $\sum_{1\leq a\leq d} \mu_a =n$. We denote $\mu\mmodels_d n$.

We denote by $H^{\mu}$ the algebra $H_{\mu_1} \otimes \ldots \otimes H_{\mu_d}$ (by convention $H_0:=\C[u^{\pm 1},v]$), where the tensor products are over $\C[u^{\pm1},v]$.\\

\paragraph{\textbf{Partitions of $\{1,\dots,n\}$.}}
 
Let $\cP_d(n)$ be the set of ordered partitions of $\{1,\dots,n\}$ into $d$ parts. An element $I\in \cP_d(n)$ consists of $d$ pairwise disjoint subsets $I_1,\dots,I_d$ of $\{1,\dots,n\}$ such that $I_1\cup\dots\cup I_d=\{1,\dots,n\}$. We denote $I=I_1\sqcup \dots\sqcup I_d$. Note that some subsets among $I_1,\dots,I_d$ can be empty.

We consider ordered partitions in the sense that $I=I_1\sqcup \dots\sqcup I_d$ and $J=J_1\sqcup \dots\sqcup J_d$ are the same element of $\cP_d(n)$ if and only if $I_a=J_a$, for $a=1,\dots,d$.

For a given $I\in\cP_d(n)$, there is an associated composition $\mu=\Comp(I)\mmodels_d n$ defined by $\mu_a=|I_a|$ for $a=1,\dots,d$. We say that $\mu$ is the \emph{shape} of $I$. For a given $\mu\mmodels_d n$, we denote by $\cP(\mu)$ the set of partitions $I$ such that $\Comp(I)=\mu$.

For a subset $S\subset \{1,\dots,n\}$ and a permutation $\pi\in\mathfrak{S}_n$, let $\pi(S):=\{\pi(s)\ |\ s\in S\}$. Thus, there is an action of the symmetric group $\mathfrak{S}_n$ on each $\cP(\mu)$. Namely, for $I\in\cP(\mu)$ and $\pi\in\mathfrak{S}_n$, $\pi(I)$ is the partition $\pi(I_1),\dots,\pi(I_d)$.

The cardinality of the set $\cP(\mu)$ is 
\begin{equation}\label{mmu}
|\cP(\mu)|=\displaystyle \frac{n!}{\mu_1!\mu_2!\dots\mu_d!}\ .
\end{equation}

\subsection{Isomorphism Theorem}

For $\mu\mmodels_d n$, we denote $\Mat_{\mu}(H^{\mu})$ the algebra of square matrices with coefficients in $H^{\mu}$ and with lines and columns indexed by elements of $\cP(\mu)$, that is, by partitions of $\{1,\dots,n\}$ of shape $\mu$. Below, we will give an element $x$ of $\bigoplus_{\mu\mmodels_d\,n}\Mat_{\mu}(H^{\mu})$, by giving all its coefficients, denoted  $x_{I,I'}$, where $I,I'\in\cP(\mu)$ and $\mu\mmodels_d n$.

Let $\{\xi_1,\dots,\xi_d\}$ be the set of complex roots of unity of order $d$. The following theorem is proved in \cite{JaPo} (see also \cite{Po}).
\begin{theorem}\label{theo-iso}
The algebra $Y_{d,n}$ is isomorphic to $\bigoplus_{\mu\mmodels_d\,n}\Mat_{\mu}(H^{\mu})$. The isomorphism $\Psi_{d,n}$ is given on the generators by the following formulas:
\begin{itemize}
\item Let $j\in\{1,\dots,n\}$. The only non-zero coefficients of $\Psi_{d,n}(t_j)$ are:
\begin{equation}\label{iso-1}
\Psi_{d,n}(t_j)_{I,I}= \xi_a\ ,
\end{equation}
where $I\in\cP_d(n)$ and $a\in\{1,\dots,d\}$ is determined by $j\in I_a$.
\item Let $i\in\{1,\dots,n-1\}$. The only non-zero coefficients of $\Psi_{d,n}(g_i^{\pm1})$ are:
\begin{equation}\label{iso-2}
\Psi_{d,n}(g_i^{\pm1})_{I,s_i(I)}= \left\{\begin{array}{ll}
u^{\pm1} & \text{if $s_i(I)\neq I$\,,}\\[0.5em]
1\otimes \dots 1\otimes \og_k^{\pm1}\otimes 1\dots \otimes 1\quad & \text{if $s_i(I) = I$\,,}
\end{array}\right.
\end{equation}
where $I\in\cP_d(n)$ and, in the second line, $\og_k^{\pm1}$ is in position $a$ determined by $i\in I_a$ and we have $k=|\{j\in I_a\ |\ j\leq i\}|$.
\end{itemize}

\end{theorem}

\begin{remark} We provide here the translation between the formulation in \cite{JaPo,Po} and the one above of the isomorphism. A complex character $\chi$ of the subalgebra $\mathcal{T}_{d,n}=\langle t_1,\dots,t_n\rangle$ of $Y_{d,n}$ is determined by the choice of $\chi (t_j)\in \{\xi_1,\ldots ,\xi_d  \}$ for each $j=1,\ldots,n$. 

We have a bijection $\chi\leftrightarrow I(\chi)$ between  the complex characters $\chi$ of $\mathcal{T}_{d,n}$ and the set $\cP_d(n)$. It is given by setting $I(\chi)_a:=\{j\in \{1,\ldots,n\}\ |\ \chi (t_j)=\xi_a\}$ for $a=1,\dots,d$. Moreover, the action of $\pi\in\mathfrak{S}_n$ on characters, defined by $\pi(\chi)\bigl(t_j\bigr):=\chi(t_{\pi^{-1}(j)})$, is compatible with the action on $\cP_d(n)$ in the following sense:  the character $\pi(\chi)$ corresponds to the partition $\pi(I)$ if $\chi$ corresponds to the partition $I$ (in other words, $I\bigl(\pi(\chi)\bigr)=\pi\bigl(I(\chi)\bigr)$).

This explains the equivalence of the formulation of Theorem \ref{theo-iso} with the one in \cite{JaPo,Po}, where the lines and columns of matrices were indexed by characters $\chi$.
\end{remark}

\section{Invariants for links}\label{Sec-inv}

Let $\gamma$ be another indeterminate. We work from now on over the ring $R:=\C[u^{\pm1},v^{\pm1},\gamma^{\pm1}]$ and we consider now all algebras over this extended ring $R$.

\subsection{Definition of the invariants}

We denote by $B_n$ the braid group on $n$ strands, and by $\sigma_1,\ldots,\sigma_{n-1}$ its generators in the Artin presentation ($\sigma_i$ corresponds to a positive elementary braiding between strands $i$ and $i+1$) . The defining relations are (\ref{def-1}) with $g_i$ replaced by $\sigma_i$.

The natural surjective morphism from the group algebra $RB_n$ to the Hecke algebra $H_n$ is given on the generators by $B_n\ni\sigma_i\mapsto \og_i\in H_n$, $i=1,\dots,n-1$. For a braid $\beta\in B_n$, we denote $\overline{\beta}$ its image in $H_n$.

There is also a natural surjective morphism from $B_n$ to the symmetric group $\mathfrak{S}_n$, given  on the generators by $B_n\ni\sigma_i\mapsto s_i\in \mathfrak{S}_n$, $i=1,\dots,n-1$. For a braid $\beta\in B_n$, we denote $p_{\beta}$ its image in $\mathfrak{S}_n$, and refer to $p_{\beta}$ as to the underlying permutation of $\beta$.\\

\paragraph{\textbf{HOMFLYPT polynomial.}} Let $\{\tau_n\ :\ H_n\to\C[u^{\pm1},v^{\pm1}]\}_{n\geq1}$ be a family of linear maps satisfying:
\[\begin{array}{ll}
\tau_n(xy)=\tau_{n}(yx) & \text{for $n\geq1$ and $x,y\in H_n$\,,}\\[0.5em]
\tau_{n+1}(x\og_{n}^{\pm1})=\tau_{n}(x) \quad& \text{for $n\geq1$ and $x\in H_n$\,,}
\end{array}\]
Fixing arbitrarily $\tau_1(1)\in\C[u^{\pm1},v^{\pm1}]$, it is well-known that such a family exists and is uniquely determined by the above conditions (\emph{e.g.} \cite[\S 4.5]{GP}). We call the family of maps $\{\tau_n\}_{n\geq1}$ the Markov trace on the Hecke algebras.

Let $L$ be a link and $\beta_L$ be a braid in $B_n$ closing to $L$. The HOMFLYPT polynomial $P(L)$ of the link $L$ is defined as follows
\[P(L):=\tau_n(\overline{\beta_L})\ .\]
The HOMFLYPT polynomial is a topological invariant of $L$ with values in $\C[u^{\pm1},v^{\pm1}]$.\\

\paragraph{\textbf{Invariants $T_d(L)$ from Yokonuma--Hecke algebras.}}
Let $d>0$. We consider the following map from $RB_n$ to $Y_{d,n}$ given on the generators by:
\[\delta_{d,n}\ :\ \ \ \sigma_i\mapsto \bigl(\gamma+(1-\gamma)e_i\bigr)g_i\ .\]
It is straightforward to check that $\delta_{d,n}$ extends to a morphism of algebras \cite[\S 6]{JaPo}.

For an element $x\in Y_{d,n}$, let $\bigl(\Psi_{d,n}(x)\bigr)_{\mu}$ be the component of $\Psi_{d,n}(x)$ in $\Mat_{\mu}(H^{\mu})$. We will denote by $\operatorname{Tr}_{\mu}$ the usual trace of a matrix in $\Mat_{\mu}(H^{\mu})$. Finally, let $\Comp_d^{>0}(n)$ be the set of compositions of $n$ into $d$ parts such that all parts are non-zero.

\begin{definition}\label{def-Td}
Let $L$ be a link and $\beta_L$ a braid in $B_n$ closing to $L$. For $d>0$, we define:
\begin{equation}\label{def-Td-form}
T_d(L):=\sum_{\mu\in\Comp_d^{>0}(n)}(\tau_{\mu_1}\otimes\dots\otimes\tau_{\mu_d})\circ\operatorname{Tr}_{\mu}\circ \Bigl(\Psi_{d,n}\bigr(\delta_{d,n}(\beta_L)\bigr)\Bigr)_{\mu}\ .
\end{equation}
\end{definition}
It is proved in \cite{JaPo} that $T_d(L)$ is an invariant of the link $L$. The invariant $T_d(L)$ takes values in $R=\C[u^{\pm1},v^{\pm1},\gamma^{\pm1}]$. An equivalent formula, using the notation of the preceding section, is
\begin{equation}\label{def-Td-form2}
T_d(L):=\sum_{I\in\cP^{>0}_d(n)}(\tau_{\mu_1}\otimes\dots\otimes\tau_{\mu_d})\Bigl(\Psi_{d,n}\bigr(\delta_{d,n}(\beta_L)\bigr)_{I,I}\Bigr)\ ,
\end{equation}
where $\cP^{>0}_d(n)$ is the subset of $\cP_d(n)$ consisting of partitions $I=I_1\sqcup\dots\sqcup I_d$ with $I_a\neq\emptyset$ for every $a=1,\dots,d$.

\begin{remarks}$\textbf{(i)}$ By definition, $T_1(L)$ is the HOMFLYPT polynomial $P(L)$.

$\textbf{(ii)}$ The Juyumaya--Lambropoulou invariants obtained from $Y_{d,n}$ in\cite{JuLa2} (see also \cite{ChJKL} and references therein) are certain combinations of the invariants $T_d(L)$ for some specialisations of the parameter $\gamma$. This is explained in details in \cite{Po}.
\end{remarks}

\section{Main result}

To fix conventions, in a braid diagram, we read from bottom to top and we number all strands of the braid with the indices it starts at the bottom. We read the braid word from left to right accordingly.
For instance, the braid word corresponding to the braid on the left of Figure \ref{erase} is $\sigma_1^{-1}\sigma_2\sigma_1^{-1}\sigma_2\sigma_1^{-1}$.\\

\paragraph{\textbf{Removing strands maps.}} 

Let $J\subset\{1,\dots,n\}$ and let $\beta\in B_n$. Starting from a braid diagram of $\beta$, erase all strands with number not belonging to $J$. We give an example of this process in Figure \ref{erase}. We denote by $F_J(\beta)$ the braid corresponding to the resulting diagram. This gives a well-defined map $F_J\ :\ B_n\to B_{|J|}$ from the braid group on $n$ strands to the braid group on $|J|$ strands.

\begin{center}
\begin{figure}
\scalebox{0.2}{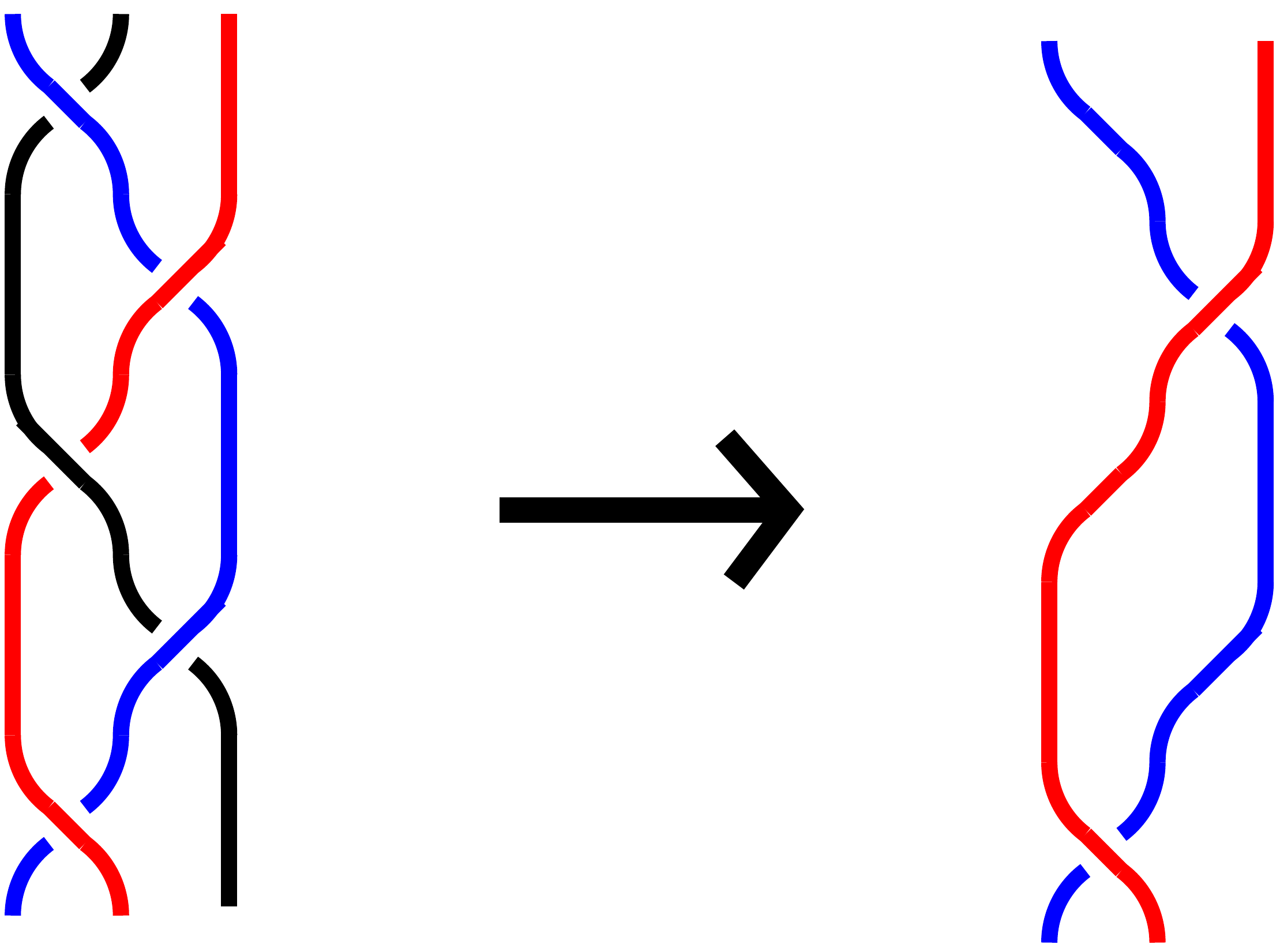}
\caption{Example of the removing map with $J=\{1,2\}$.}
\label{erase}
\end{figure}
\end{center}

Let $\beta\in B_n$ and $i\in\{1,\dots,n-1\}$. It is immediate to check that the following formulas are satisfied:
\begin{equation}\label{form-F}
F_J(1)=1\ \ \ \ \ \ \text{and}\ \ \ \ \ \ \ F_J(\beta\sigma_i^{\pm1})=\left\{\begin{array}{ll}
F_J(\beta)\sigma_k^{\pm1}\quad & \text{if $\{i,i+1\}\subset J$\,,}\\[0.5em]
F_{s_i(J)}(\beta) & \text{otherwise,}
\end{array}\right.
\end{equation}
where, if $\{i,i+1\}\subset J$ then the index $k$ is $|\{j\in J\ |\ j\leq i\}|$. 
Note that if $i\notin J$ and $i+1\notin J$ then $F_J(\beta\sigma_i^{\pm1})=F_J(\beta)$.\\

\paragraph{\textbf{Linking numbers.}} Let $J\subset\{1,\dots,n\}$. In a braid diagram, a crossing is called mixed (relative to $J$) whenever it involves a strand with number in $J$ with a strand with number not in $J$. Let $\beta\in B_n$ and denote by $l_J(\beta)$ the number of positive mixed crossing minus the number of negative mixed crossing. This gives a well-defined map $l_J\ :\ B_n\to \Z$.

It is immediate to check that the following formulas are satisfied:
\begin{equation}\label{form-lJ}
l_J(1)=0\ \ \ \ \ \ \text{and}\ \ \ \ \ \ \ l_J(\beta\sigma_i^{\pm1})=\left\{\begin{array}{ll}
l_J(\beta) & \text{if $s_i(J)=J$\,,}\\[0.5em]
l_J(\beta)\pm1\quad &  \text{otherwise,}
\end{array}\right.
\end{equation}
where $\beta\in B_n$ and $i\in\{1,\dots,n-1\}$.

Finally we set, for $I=I_1\sqcup\dots\sqcup I_d\in\cP_d(n)$ and $\beta\in B_n$,
\begin{equation}\label{form-lI}
l_I(\beta):=\frac{l_{I_1}(\beta)+\dots+l_{I_d}(\beta)}{2}\ .
\end{equation}

\paragraph{\textbf{Main result.}}We are now ready to formulate the main result on the isomorphism.
\begin{theorem}\label{prop1}
Let $\beta\in B_n$. The only non-zero coefficients in $\Psi_{d,n}\bigl(\delta_{d,n}(\beta)\bigr)$ are
\[\Psi_{d,n}\bigl(\delta_{d,n}(\beta)\bigr)_{p_{\beta}(I),I}=(u\gamma)^{l_I(\beta)}\overline{F_{I_1}(\beta)}\otimes\dots\otimes\overline{F_{I_d}(\beta)}\ ,\]
where $I\in\cP_d(n)$.
\end{theorem}
\begin{proof}
Let $i\in\{1,\dots,n-1\}$ and $I=I_1\sqcup\dots\sqcup I_d\in\cP_d(n)$. The formulas in Theorem \ref{theo-iso} imply that the coefficient $\Psi_{d,n}(e_i)_{I,I}$ is equal to 1 if $i$ and $i+1$ belongs to the same component of $I$ (that is, if $s_i(I)=I)$). And moreover all other coefficients of $\Psi_{d,n}(e_i)$ are zero. Therefore, the only non-zero coefficients in $\Psi_{d,n}\bigl(\delta_{d,n}(\sigma_i)\bigr)$ are
\begin{equation}\label{Psidelta-g}
\Psi_{d,n}\bigl(\delta_{d,n}(\sigma_i)\bigr)_{I,s_i(I)}= \left\{\begin{array}{ll}
u\gamma & \text{if $s_i(I)\neq I$\,,}\\[0.5em]
1\otimes \dots 1\otimes \og_k\otimes 1\dots \otimes 1\quad & \text{if $s_i(I) = I$\,,}
\end{array}\right.
\end{equation}
where $I\in\cP_d(n)$ and, in the second line, $\og_k$ is in position $a$ determined by $i\in I_a$ and $k=|\{j\in I_a\ |\ j\leq i\}|$.

Let $\beta\in B_n$. The proposition is trivial if $\beta=1$. So let $\beta=\beta'\sigma_i^{\pm1}$ with $\beta'\in B_n$ and $i\in\{1,\dots,n-1\}$, and assume that the proposition is satisfied for $\beta'$. Let $I,I'\in\cP_d(n)$ of the same shape. As $\Psi_{d,n}$ and $\delta_{d,n}$ are morphisms of algebras, we have
\[\Psi_{d,n}\bigl(\delta_{d,n}(\beta)\bigr)_{I',I}=\sum_{I''}\Psi_{d,n}\bigl(\delta_{d,n}(\beta')\bigr)_{I',I''}\Psi_{d,n}\bigl(\delta_{d,n}(\sigma_i^{\pm1})\bigr)_{I'',I}\ ,\]
where the sum is over $I''\in\cP_d(n)$ of the same shape than $I,I'$. Formula (\ref{Psidelta-g}) shows in particular that the only term in the sum which can be non-zero is obtained for $I''=s_i(I)$. Then the theorem applied to $\beta'$ shows that this term is actually non-zero only if $I'=p_{\beta'}s_i(I)$. As $p_{\beta'}s_i=p_{\beta}$, we obtain that the coefficient $\Psi_{d,n}\bigl(\delta_{d,n}(\beta)\bigr)_{I',I}$ is non-zero only if $I'=p_{\beta}(I)$. Moreover, in this case, we have
\[\Psi_{d,n}\bigl(\delta_{d,n}(\beta)\bigr)_{p_{\beta}(I),I}=\Psi_{d,n}\bigl(\delta_{d,n}(\beta')\bigr)_{p_{\beta'}s_i(I),s_i(I)}\Psi_{d,n}\bigl(\delta_{d,n}(\sigma_i^{\pm1})\bigr)_{s_i(I),I}\ ,\]
We conclude the calculation using the theorem applied to $\beta'$ and Formula (\ref{Psidelta-g}). First if $s_i(I)\neq I$, we obtain
\[\Psi_{d,n}\bigl(\delta_{d,n}(\beta)\bigr)_{p_{\beta}(I),I}=(u\gamma)^{l_{s_i(I)}(\beta')\pm1}\overline{F_{s_i(I_1)}(\beta')}\otimes\dots\otimes\overline{F_{s_i(I_d)}(\beta')}\ .\]
In this case, using the properties of the maps $F_J$ and $l_J$ recalled above, we have $F_{I_a}(\beta)=F_{s_i(I_a)}(\beta')$ for every $a=1,\dots,d$, and $l_I(\beta)=l_{s_i(I)}(\beta')\pm1$. So we obtain the desired formula.

On the other hand, if $s_i(I)=I$, we have
\[\Psi_{d,n}\bigl(\delta_{d,n}(\beta)\bigr)_{p_{\beta}(I),I}=(u\gamma)^{l_{I}(\beta')}\overline{F_{I_1}(\beta')}\otimes\dots\otimes\overline{F_{I_a}(\beta')}\overline{g_k}^{\pm1}\otimes\dots\otimes\overline{F_{I_d}(\beta')}\ ,\]
where $a$ is such that $i\in I_a$ and $k=|\{j\in I_a\ |\ j\leq i\}|$. In this case, we have $F_{I_b}(\beta)=F_{I_b}(\beta')$ for every $b\neq a$, $F_{I_a}(\beta)=F_{I_a}(\beta')\sigma_k^{\pm1}$ and $l_I(\beta)=l_{I}(\beta')$. So we obtain the desired formula as well.
\end{proof}

\paragraph{\textbf{Consequences for the invariants $T_d$.}} 
Let $\beta\in B_n$ closing to a link $L$ and let $I\in\cP_d(n)$. Assume that $p_{\beta}(I)=I$. Let $(i_1i_2\dots i_k)$ (with $k\geq 1$) be a cyclic permutation appearing in the decomposition of $p_{\beta}$ as a product of disjoint cycles. The condition $p_{\beta}(I)=I$ then means that for every such cycle, there exists $a\in\{1,\dots,d\}$ such that $i_1,i_2,\dots,i_k\in I_a$.

Recall that $F_{I_a}(\beta)$ is defined as the braid obtained from $\beta$ by keeping only the strands with indices in $I_a$. We just explained that $I_a$ is a (possibly empty) union of some cycles of $p_{\beta}$. This is equivalent to say that, for every $a\in\{1,\dots,d\}$, the closure of the braid $F_{I_a}(\beta)$ is a (possibly empty) union of connected components of $L$ (\emph{i.e.} a sublink of $L$).

As a conclusion, there is a bijection
\[\{\ I\in\cP_d(n)\ \text{such that}\ p_{\beta}(I)=I\ \}\ \ \ \ \longleftrightarrow\ \ \ \ (L_1,\dots,L_d)\ ,\]
between the set of $I\in\cP_d(n)$ such that $p_{\beta}(I)=I$, and the set of complete $d$-tuples $(L_1,\dots,L_d)$ of sublinks of $L$ (complete means that the union is $L$). The map from left to right is given by $L_a=F_{I_a}(\beta)$ as explained above. The inverse map is obtained by defining $I_a$ to be the set of indices of the strands forming $L_a$. 

Finally, we note that under the bijection above, the integer $l_{\beta}(I)$ becomes by definition 
\[l(L_1,\dots,L_d):=\sum_{1\leq i<j\leq d}\text{lk}(L_i,L_j) \ ,\]
where $\text{lk}(L_i,L_j)$ is the linking number between components $L_i$ and $L_j$. Moreover, the set $\cP_d^{>0}(n)$ of Formula (\ref{def-Td-form2}) corresponds to complete $d$-tuples of non-empty sublinks of $L$.

With these remarks, it remains only to combine the defining formula (\ref{def-Td-form2}) for the invariant $T_d(L)$ together with Theorem \ref{prop1} to obtain the following corollary announced in the introduction.
\begin{corollary}\label{coro-fin}
Let $L$ be a link and $d>0$.
\begin{equation}
T_d(L)=d!\sum_{\{L_1,\dots,L_d\}}(u\gamma)^{l(L_1,\dots,L_d)}P(L_1)\dots P(L_d)\ ,
\end{equation}
where the sum is over the set of all complete families $\{L_1,\dots,L_d\}$ of $d$ distinct non-empty sublinks of $L$.
\end{corollary}

\begin{remarks} \textbf{(i)} In particular, if $d>N$, where $N$ is the number of connected components of $L$, then $T_d(L)=0$. This was already proved in \cite{Po}.

\textbf{(ii)} Let $K_1,\dots,K_N$ be the connected components of $L$. Then in $T_N(L)$, there is only one term in the sum, which is
\[T_N(L)=N!(u\gamma)^{l(K_1,\dots,K_N)}P(K_1)\dots P(K_N)\ .\]

\textbf{(iii)}  It was proved in \cite{ChJKL,JaPo}, that the invariants obtained from the Yokonuma--Hecke algebras are topologically equivalent to the HOMFLYPT polynomial for knots, but not for links.
The case $d=1$ in the main theorem recovers the case for knots, but also indicates what to expect in the case of links. Looking quickly to the pair of links founded in \cite{ChJKL}, one realizes shortly that some of them are distinguished by their linking matrix whereas, some others by their sublinks which are detected by the HOMFLYPT polynomials. Let us mention here that the question of determining the linking numbers from the HOMFLYPT polynomial was studied by  Jozef Przytycki \cite{Przy} and by Adam Sikora in its Master degree thesis \cite{Sik} and that an early example of pairs of links with the same HOMFLYPT polynomial but different linking matrix was given by Joan Birman in \cite{Bir}.

\end{remarks}

\end{document}